\documentclass[12pt]{article}

\usepackage{amsmath, amsthm,  amsfonts, amssymb, mathrsfs,eufrak,textcomp}
\usepackage{color}
\usepackage{tikz}
\usetikzlibrary{calc, intersections}
\usetikzlibrary{chains}
\usepackage{epstopdf}
\usepackage[all]{xy}

\definecolor{pdflinkcolor}{rgb}{.1,.1,.6}	
\definecolor{pdfcitecolor}{rgb}{.6,.1,.1}	
\definecolor{pdfanchorcolor}{rgb}{0,1,0}	
\definecolor{pdfurlcolor}{rgb}{.1,.6,.1}	
\definecolor{pdfpagecolor}{rgb}{0,0,1}		
\definecolor{pdffilecolor}{rgb}{1,0,0}		
\usepackage[%
   breaklinks,
   colorlinks=true, 
   citecolor=pdfcitecolor,
   linkcolor=pdflinkcolor,
   anchorcolor=pdfanchorcolor,
   urlcolor=pdfurlcolor,
   pagecolor=pdfpagecolor,
   filecolor=pdffilecolor,
   pagebackref=false,
   pdfauthor={C. Segovia, M. Winklmeier},
   pdftitle={Combinatorial computations in cobordism categories}]{hyperref}

\newtheorem{theorem}{Theorem}[section]

\newtheorem{lemma}[theorem]{Lemma}
\newtheorem{cor}[theorem]{Corollary}

\theoremstyle{definition}

\newtheorem{example}[theorem]{Example}

\theoremstyle{remark}

\numberwithin{equation}{section}
\newcommand{\op}{\operatorname}
\newcommand{\cob}{\mathscr{S}}

\newcommand{\id}{\mathrm{id}}

\def\ZZ{\mathbb{Z}}
\def\NN{\mathbb{N}}
\def\CC{\mathbb{C}}

\DeclareMathOperator{\SL}{SL}

\title{\bf On the density of certain languages with $p^2$ letters}
\author{
Carlos Segovia\footnote{
Mathematisches Institut,
Universit\"at Heidelberg, Germany,\newline
csegovia@mathi.uni-heidelberg.de}\, and\,
Monika Winklmeier\footnote{
Departamento de Matem\'aticas, Universidad de los Andes, Colombia,\newline
mwinklme@uniandes.edu.co}
}

\date{July 16, 2015}

\begin{document}
\maketitle

\begin{abstract}
   The sequence $(x_n)_{n\in\NN} = (2,5,15,51,187,\ldots)$ given by the rule $x_n=(2^n+1)(2^{n-1}+1)/3$ appears in several seemingly unrelated areas of mathematics.
For example, $x_n$ is the density of a language of words of length $n$ with four different letters.
It is also the cardinality of the quotient of $(\ZZ_2\times \ZZ_2)^n$ under the left action of the special linear group $\SL(2,\ZZ)$.
In this paper we show how these two interpretations of $x_n$ are related to each other.
More generally, for prime numbers $p$ we show a correspondence between a quotient of $(\ZZ_p\times\ZZ_p)^n$ and a language with $p^2$ letters and words of length~$n$.
\medskip

\end{abstract}

\section{Introduction}
When we looked for the sequence, listed as sequence \href{https://oeis.org/A007581}{A007581} in the Online Encyclopedia of Integer Sequences \cite{oeis},
\begin{align}
\label{eq:sequence}
(x_n)_{n\in\NN}
= \left( (2^n+1)(2^{n-1}+1)/3 \right)_{n\in\NN}
= (2,\,5,\,15,\,51,\,187,\,\ldots)
\end{align}
we were greatly surprised to find that it arises in many different contexts.
Indeed, there exist at least four distinct areas where the sequence appears:

\newcommand\mylabelenumi[1]{{\upshape (#1)}}
\newcommand\enumiref[1]{\mylabelenumi{\ref{#1}}}
\renewcommand\labelenumi{\mylabelenumi{\theenumi}}

\begin{enumerate}

   \item\label{ca3}
   the density of a language with four letters (see \cite{mor});

   \item\label{ca1}
   the dimension of the universal embedding of the symplectic dual polar space (see \cite{brow});

   \item\label{ca2}
   the number of isomorphism classes of regular fourfold coverings of a graph with Betti number $n$ and with voltage group $\ZZ_2 \times \ZZ_2$ (see \cite{jin});

   \item\label{ca4}
   the rank of the $\ZZ_2^n$-cobordism category in dimension 1+1 (see \cite{carlos1}).

\end{enumerate}

In combinatorial theory, \enumiref{ca1} was known as Brouwer's conjecture.
It was proved by P.~Li (see \cite{li}) and independently by A.~Blokhuis and A.~E.~Brouwer (see \cite{brow}).
This former conjecture states that the dimension of the universal embedding of the symplectic polar space $\mathrm{Sp}_{2n}(2)$ is $(2^n+1)(2^{n-1}+1)/3$.

In \cite{carlos2}, the first author constructs a function which relates
\eqref{ca1} and \eqref{ca3}.
In this paper, in Theorem~\ref{thm:fundgroup}, we establish a relation between 
\enumiref{ca3} and
\enumiref{ca4}.
The relation between \enumiref{ca2} and \enumiref{ca4} will be considered in a future work.
\smallskip

The paper is organised as follows.
In Section~\ref{sec1} we briefly describe
\enumiref{ca3} and \enumiref{ca4}
and give yet another interpretation of the sequence in Section~\ref{subsec:Z} as the number of orbits of $(\ZZ_2\times\ZZ_2)^n$ under the left action of $\SL(2,\ZZ)$.
In Section~\ref{sec:bijection} we give a bijection between these interpretations.
Finally, in Section~\ref{sec3} we show that the sequence \eqref{eq:sequence} appears also
in certain point-set geometries, in cobordism categories and in topological field theory.
\smallskip

Throughout the paper, the cardinality of a finite set $M$ is denoted by $|M|$.

\section{Interpretations of the sequence}\label{sec1}
In this section we will present three instances where the sequence
\eqref{eq:sequence} naturally appears.
For other interpretations of the sequence and their relation to the ones given above, we refer to \cite{carlos2} and \cite[sequence A007581]{oeis}.

\subsection{Density of a language}

For a prime number $p$ we consider a language with $p^2$ letters $0,1,2,\dots,
p^2 -1$.
For $n\in\NN$ we define the set $W_p^n$ as the set of words $a_1 a_2 \dots a_n$
of length $n$ such that there exist $1\le j<k$ with:
\begin{itemize}
   \item[(R1)] $a_i = 0$ for $i < j$,
   \item[(R2)] $a_j = 1$,
   \item[(R3)] $a_i \in\{0, 1, \dots, p-1\}$ for $j<i<k$,
   \item[(R4)] $a_k \in\{p, 2p, 3p, \dots, (p-1)p\}$ if $k\le n$,
   \item[(R5)] $a_i \in\{0, 1, \dots, p^2-1\}$ for $i>k$ if $k< n$.
\end{itemize}
\smallskip

\noindent
We call $|W_p^n|$ the \emph{density} of the language with $p^2-1$ letters and words of length~$n$.

\begin{example}[Moreira and Reis \cite{mor}]
Let us consider the case $p=2$.
The letters of our language are $0,1,2,3$ and a word $a_1\dots a_n$ belongs
to $W_2^n$ if and only if its letters satisfy $0\le a_i\le \max_{j< i}\{a_j\}+1$.
For instance,
\begin{itemize}
   \item
   $W_2^1 = \{0, 1\}$,
   \item
   $W_2^2 = \{00, 01, 10, 11, 12 \}$,
   \item
   $W_2^3 = \{000, 001, 010, 100, 011, 110, 101, 111, 012, 112, 121, 122, 120, 102, 123 \}$.
\end{itemize}
Observe that
$|W_2^1| = 2 = x_1$,
$|W_2^2| = 5 = x_2$,
and $|W_2^3| = 15 = x_3$,
where $x_n$ is the sequence defined in \eqref{eq:sequence}.
In Corollary~\ref{cor:MoreiraReisFormula} we will prove that
$|W_2^n| = x_n$ for all $n\in\NN$.

\end{example}

\subsection{$(\ZZ_p\times \ZZ_p)^n/\SL(2,\ZZ)$}
\label{subsec:Z}
Let us consider the usual left action of $\SL(2,\ZZ)$ on the vector space $(\ZZ_p\times\ZZ_p)^n$.
Vectors in this vector space are denoted by
$\begin{pmatrix}u\\v
\end{pmatrix}$
or $(u,v)^t$ where
$u = (u_1, u_2, \dots, u_n),\
v = (v_1, v_2, \dots, v_n)\in \ZZ_p^n$ are line vectors.
%
%
Observe that $\SL(2,\ZZ)$ is generated by the matrices
$\begin{pmatrix} 0& 1 \\ -1 & 0
\end{pmatrix}$
   and
$\begin{pmatrix} 1& 1 \\ 0 & 1
\end{pmatrix}$.
Two elements
$(u,v)^t,\, (x,y)^t \in (\ZZ_p\times\ZZ_p)^n$
belong to the same orbit if they can be transformed into each other by line transformations given by elements of $\SL(2,\ZZ)$, that is we may change one line with the negative of another one,
or sum a multiple of one line to another.
We call $(u,v)^t,\, (x,y)^t\in (\ZZ_p\times\ZZ_p)^n$ \emph{equivalent} if they belong to the same orbit. Clearly this gives an equivalence relation.
The equivalence class or the orbit containing $(u,v)^t$ is denoted by $[(u,v)^t]$.
Note that any two given orbits are either equal or disjoint.

Let us consider the example $p=2$.
For $n=1$, there are the two orbits $\{(0,0)^t\}$ and $\{ (1,0)^t, (0,1)^t, (1,1)^t \}$.
For $n=2$, there are the five orbits
\begin{align*} 
   & \left\{
   \begin{pmatrix}
      0&0 \\ 0&0
   \end{pmatrix}
   \right\},
   \\[1ex]
   & \left\{
   \begin{pmatrix}
      0&0 \\ 0&1
   \end{pmatrix},
   \begin{pmatrix}
      0&1 \\ 0&0
   \end{pmatrix},
   \begin{pmatrix}
      0&1 \\ 0&1
   \end{pmatrix}
   \right\},
   \\[1ex]
   & \left\{
   \begin{pmatrix}
      1&0 \\ 0&0
   \end{pmatrix},
   \begin{pmatrix}
      0&0 \\ 1&0
   \end{pmatrix},
   \begin{pmatrix}
      1&0 \\ 1&0
   \end{pmatrix}
   \right\},
   \\[1ex]
   & \left\{
   \begin{pmatrix}
      1&1 \\ 0&0
   \end{pmatrix},
   \begin{pmatrix}
      0&0 \\ 1&1
   \end{pmatrix},
   \begin{pmatrix}
      1&1 \\ 1&1
   \end{pmatrix}
   \right\},
   \\[1ex]
   & \left\{
   \begin{pmatrix}
      1&0 \\ 0&1
   \end{pmatrix},
   \begin{pmatrix}
      1&1 \\ 0&1
   \end{pmatrix},
   \begin{pmatrix}
      1&0 \\ 1&1
   \end{pmatrix},
   \begin{pmatrix}
      0&1 \\ 1&0
   \end{pmatrix},
   \begin{pmatrix}
      0&1 \\ 1&1
   \end{pmatrix},
   \begin{pmatrix}
      1&1 \\ 1&0
   \end{pmatrix}
   \right\}.
\end{align*}
Note that for these examples the number of orbits coincides with the number of words of the type described in the section before.
In Section~\ref{sec:bijection} this will be proved rigorously.

\section{Bijection between $W_p^n$ and $(\ZZ_p\times \ZZ_p)^n/\SL(2,\ZZ)$} 
\label{sec:bijection}

In this section we construct a bijection from $W_p^n$ to $(\ZZ_p\times \ZZ_p)^n/\SL(2,\ZZ)$.
As a corollary we obtain that both sets have the same cardinality whose value will be calculated in the next section.

\begin{theorem}
   \label{thm:bijection}
   Let $p$ be a prime number and $n\in\NN$.
   Then there exists a bijection between the sets $W_p^n$ and $(\ZZ_p\times\ZZ_p)^n/\SL(2,\ZZ)$.
\end{theorem}

\begin{proof}
   The language under consideration has $p^2$ letters denoted by $0, 1, 2, \dots, p^2-1$.
   Let us define
   $\phi: \{0,1,\dots, p^2-1\} \to \ZZ_p\times \ZZ_p$
   by $\phi(a) =
   \begin{pmatrix}
      u\\ v
   \end{pmatrix}$
   where $u,v$ are the unique elements in $\{0,1,\dots, p-1\}$ such that $a = u + vp$.
   Clearly, $\phi$ is a bijection.
   Now let us define
   \begin{align*}
      f_n &: W_p^n \to (\ZZ_p\times \ZZ_p)^n,
      \quad
      f_n(a_1a_2\dots a_n) = (\phi(a_1), \phi(a_2),\dots, \phi(a_n))
   \end{align*}
   and
   \begin{align*}
      [f_n] &: W_p^n \to (\ZZ_p\times \ZZ_p)^n/\SL(2,\ZZ),
      \quad
      [f_n](a_1a_2\dots a_n) = [f_n(a_1a_2\dots a_n)].
   \end{align*}
   Clearly $f_n$ is well-defined.
   We show that $[f_n]$ is an injection.
   Since a word in $W_p^n$ must satisfy the rules (R1)-(R5), the range of $f_n$ consists exactly of $0$ and the vectors of the form
   \begin{gather}
      \label{eq:v1}
      \begin{pmatrix}
	 0 & \dots\;  0 & 1 & u_{j+1} & \dots & u_{k-1} & 0     & u_{k+1} & \dots &u_n\\
	 0 & \dots\;  0 & 0 & 0       & \dots & 0       & v_{k} & v_{k+1} & \dots &v_n
      \end{pmatrix}
      \intertext{ or }
      \label{eq:v2}
      \begin{pmatrix}
	 0 & \dots & 0 & 1 & u_{j+1} & \dots & u_{n}\\
	 0 & \dots & 0 & 0 & 0       & \dots & 0
      \end{pmatrix}
   \end{gather}
   with $1\le j < k \le n$, $v_k\in\{1,\dots, p-1\}$ and
   $u_\ell, v_i\in \{0, 1, \dots, p-1\}$
   for $\ell\in \{j+1, \dots, n\}\setminus\{k\}$ and $i\in\{k+1,\dots n\}$.
   Let $(u,0)^t$ be of the form \eqref{eq:v2}.
   It is easy to see that in this case
   \begin{align*}
      \left[
      \begin{pmatrix} u\\ 0
      \end{pmatrix}
      \right]
      = \left\{
      \begin{pmatrix} mu\\ nu
      \end{pmatrix}
      \ :\ m,n = 0, 1, \dots, p-1
      \right\}.
   \end{align*}
   This shows that $(u_1, 0)^t,\ (u_2,0)^t$ of the form \eqref{eq:v2} belong to the same equivalence class if and only if $u_1=u_2$ and that they are not equivalent to any element of the form \eqref{eq:v1}.
   Now let $(u,v)^t, (y,z)^t$ of the form \eqref{eq:v1} and assume that they are equivalent.
   Then there exists a matrix
   $S= \begin{pmatrix} a & b \\ c & d
   \end{pmatrix} \in\SL(2,\ZZ)$
   such that $S(u, v)^t = (y, z)^t$.
   Without restriction we may assume $a,b, c,d\in \{0,1,\dots, p-1\}$.
   Since $S$ leaves zero columns invariant, the number of leading columns with only zeros in both vectors must be equal.
   Now, looking at the first non-zero column of the vectors, we find that $S(1,0)^t = (1,0)^t$. 
   Consequently we find $a=1$ and $c=0$.
   Since $\det S =1$, it follows that $d=1$.
   So we showed that
   $\begin{pmatrix} 1 & b \\ 0 & 1
   \end{pmatrix}
   \begin{pmatrix} u\\ v
   \end{pmatrix}
   =
   \begin{pmatrix} y\\ z
   \end{pmatrix}$
   for some $b$.
   In particular, $v=z$ and $u_k=y_k$, where $k$ is as indicated in \eqref{eq:v1}.
   For the $k$th column we obtain
   $\begin{pmatrix} b v_k\\ v_k
   \end{pmatrix}
   =
   \begin{pmatrix} 0\\ z_k
   \end{pmatrix}$.
   Since $v_k\neq 0$, it follows that $b=0$.
   So $S=\id$ and $(u, v)^t = (y, z)^t$.
   We showed that the equivalence classes generated by elements of the form \eqref{eq:v1} and \eqref{eq:v2} are mutually disjoint, so $f_n$ is injective.

   In order to prove that $f_n$ is surjective, we show that every nonzero $(u,v)^t\in (\ZZ_p\times\ZZ_p)^n$ belongs to an equivalence class generated by an element of the form
   \eqref{eq:v1} or \eqref{eq:v2}.
   Let $j\in\{0,\dots, n\}$ such that
   $\begin{pmatrix}
      u\\v
   \end{pmatrix}
   =
   \begin{pmatrix}
      0& \dots & 0 & u_{j} & \dots & u_n \\
      0& \dots & 0 & v_{j} & \dots & v_n
   \end{pmatrix}$
   with $(u_j, v_j)^t \neq (0,0)^t$.
   Without restriction we may assume that $u_j\neq 0$ (note that $(u,v)$ and $(-v,u)^t$ belong to the same equivalence class).
   Choose $m,n\in\ZZ$ such that $m u_j + v_j \equiv 1 \bmod p$ and $u_j+n \equiv 1\bmod p$ and set
   $S_1 =
   \begin{pmatrix}
      1+mn & n \\ m - mn - 1 & 1-n
   \end{pmatrix}$.
   Then $S_1\in\SL(2,\ZZ)$ and
   \begin{align*}
      S_1
      \begin{pmatrix}
	 u\\v
      \end{pmatrix}
      =
      \begin{pmatrix}
	 0& \dots & 0 & 1 & u^1_{j+1} & \dots & u^1_n \\
	 0& \dots & 0 & 0 & v^1_{j+1} & \dots & v^1_n
      \end{pmatrix}.
   \end{align*}
   If $v^1_{i}=0$ for all $i\ge j+1$, this is an element of the form \eqref{eq:v2}.
   Otherwise there is a $k\in\{ j+1, \dots, n\}$ such that $v^1_{i}=0$ for all $i\le k-1$ and $v_k\neq 0$.
   Choose $r\in\ZZ$ such that $u_k+r v_k \equiv 0 \bmod p$ and set
   $S_2 =
   \begin{pmatrix}
      1 & r \\ 0 & 1
   \end{pmatrix}$.
   Clearly $S_2S_1\in\SL(2,\ZZ)$ and
   \begin{align*}
      S_2 S_1
      \begin{pmatrix}
	 u\\v
      \end{pmatrix}
      =
      \begin{pmatrix}
	 0& \dots & 0 & 1 & u^1_{j+1} & \dots & u^1_{k-1} & 0     & u^2_{l+1} & \dots\ u^2_n\\
	 0& \dots & 0 & 0 & 0         & \dots & 0         & v^1_k & v^2_{l+1} & \dots\ u^2_n
      \end{pmatrix}
   \end{align*}
   is an element of the form \eqref{eq:v2}.
\end{proof}

\section{Formula for $|(\ZZ_p\times\ZZ_p)^n/\SL(2,\ZZ)|$} 

By Theorem~\ref{thm:bijection}, we know that $|(\ZZ_p\times\ZZ_p)^n/\SL(2,\ZZ)| = |W_p^n|$.
Moreira and Reis \cite{mor} proved that the density of a language with four letters has the value $(2^n+1)(2^{n-1}+1)/3$.
Theorem~\ref{thm:formula} in this section generalises their result.
\smallskip

\noindent
In the following we set $r(p,n) := |(\ZZ_p\times\ZZ_p)^n/\SL(2,\ZZ)|$ and
$F(p,n) := r(p, n+1) - r(p,n)$ for $n\ge 1$.

\begin{lemma}
   \label{lem:r}
   For any prime number $p$ we have $r(p,1)= 2$ and $r(p,2) = 2p+1$.
\end{lemma}
\begin{proof}
   The first assertion is clear.
   For the second assertion note that $(\ZZ_p\times\ZZ_p)^2$ is the disjoint union of $0$ and the orbits generated by vectors of the form \eqref{eq:v1} and \eqref{eq:v2}. Hence all equivalence classes are given by
   \begin{align*}
      \left[
      \begin{pmatrix}
	 0 & 0 \\ 0 & 0
      \end{pmatrix}
      \right],\
      \left[
      \begin{pmatrix}
	 0 & 1 \\ 0 & 0
      \end{pmatrix}
      \right],\
      \left[
      \begin{pmatrix}
	 1 & m \\ 0 & 0
      \end{pmatrix}
      \right],\
      \left[
      \begin{pmatrix}
	 1 & 0 \\ 0 & n
      \end{pmatrix}
      \right]
   \end{align*}
   with $m\in\{0, 1, \dots, p-1\},\ n\in\{1, \dots, p-1\}$ which proves the formula for $r(p,2)$.
\end{proof}
\begin{lemma}
   \label{lem:F}
   $F(p,n)=p^{n-1}(p^n+p-1)$
   for any prime number $p$ and $n\in\NN$.
\end{lemma}

\begin{proof}
   Note that the formula holds for $n=1$ because $F(1) = r(p,2)-r(p,1) = 2p-1$ by Lemma~\ref{lem:r}.
   We will prove that
   \begin{align}
      \label{eq:ind}
      F(p,n) = p F(p, n-1) + p^{2n-2}(p-1),
      \qquad n\in\NN.
   \end{align}
   The conclusion follows then easily by an induction on $n$.
   \smallskip

   Recall that $F(p,n) = r(p,n+1)-r(p, n)$ and observe that
   $(\ZZ_p\times\ZZ_p)^{n}$ can be viewed as the subspace of all vectors in $(\ZZ_p\times\ZZ_p)^{n+1}$ whose first column consists of $0$ only.
   Since $\SL(2,\ZZ)$ leaves zero columns invariant, $F(p,n)$ is the number of different orbits in $(\ZZ_p\times\ZZ_p)^{n+1}$ whose representatives have nonvanishing first column.
   Note that if the first column is not vanishing, we can always choose a representative whose first column is $(1,0)^t$.

   Now let $(u,v)^t$ be a vector in $(\ZZ_p\times\ZZ_p)^{n+1}$ with nonvanishing first column.
   Then this vector belongs to exactly one of the following cases:
   \begin{enumerate}
      \item
      The second column is vanishing.
      Then there exists $S\in\SL(2,\ZZ)$ such that
      $S \begin{pmatrix} u \\ v
      \end{pmatrix}=
      \begin{pmatrix} 1 & 0 & \dots\\ 0 & 0 & \dots
      \end{pmatrix}$.
      Clearly, the orbits generated by such vectors in $(\ZZ_p\times\ZZ_p)^{n+1}$ with nonvanishing first column and vanishing second column correspond bijectively to the orbits of vectors in $(\ZZ_p\times\ZZ_p)^{n}$ with nonvanishing first column.
      Hence there are $F(p, n-1)$ distinct orbits of this type.

      \item
      The second column is nonvanishing and there exists $S\in\SL(2,\ZZ)$ such that
      $S \begin{pmatrix} u \\ v
      \end{pmatrix}=
      \begin{pmatrix} a & 1 & \dots\\ 0 & 0 & \dots
      \end{pmatrix}$
      with some $a\in\{1, 2, \dots, p-1\}$.
      For fixed $a$, there exists a bijection between orbits generated such an elements and the orbits generated by vectors with nonvanishing first column in $(\ZZ_p\times\ZZ_p)^{n}$.
      Clearly orbits with different $a$'s are disjoint (see the proof of Theorem~\ref{thm:bijection}), we have
      $(p-1)F(p, n-1)$ different orbits generated by elements of this type.

      \item
      The second column is nonvanishing and there exists $S\in\SL(2,\ZZ)$ such that
      $S \begin{pmatrix} u \\ v
      \end{pmatrix}=
      \begin{pmatrix} 0 & 1 & \dots\\ a & 0 & \dots
      \end{pmatrix}$
      with some $a\in\{1, 2, \dots, p-1\}$.
      It is easy to see that such vectors with different $a$'s generate disjoint orbits, and that vectors with the same $a$ but different tails (the last $n-1$ columns) also generate disjoint orbits.
      Since there are $(p^2)^{n-1}$ different tails, the number of orbits generated by vectors of the form above is
      $(p-1)p^{2n-2}$.
   \end{enumerate}

   In total the number of orbits generated by vectors in $(\ZZ_p\times \ZZ_p)^{n+1}$ with nonvanishing first column is
   $F(p, n-1) + (p-1) F(p, n-1) + (p-1)p^{2n-2} = pF(p, n-1) + (p-1)p^{2n-2}$.
\end{proof}

\begin{theorem}\label{thm:formula}
   For any prime number $p$ and $n\in\NN$ we have
   \begin{equation}
      r(p,n)=\frac{p^{2n-1}+p^{n+1}-p^{n-1}+p^2-p-1}{p^2-1}\,.
   \end{equation}
\end{theorem}
\begin{proof}
   As before let $F(p,n) = r(p, n+1) - r(p,n)$.
   By Lemma~\ref{lem:F} we have that
   \begin{equation}
      \label{from1}
      F(p,n)=p^{n-1}(p^n+p-1)\,.
   \end{equation}
   This implies
   \begin{align*}
      r(p,n)
     &  = r(p,1) + F(p,1) + \dots + F(p,n-1)
     = 2 + \sum_{j=1}^{n-1} p^{j-1}(p^j+p-1)
     \\
     & = 2 + p\sum_{j=0}^{n-2} p^{2j} + (p-1)\sum_{j=0}^{n-2} p^{j}
     = 2 + \frac{p(1-p^{2n-2})}{ 1- p^2} - (1-p^{n-1})
     \\
     & =
     \frac{p^{2n-1} + p^{n+1} - p^{n-1} + p^2 - p -1}{p^2-1}.
     \qedhere
   \end{align*}
\end{proof}

\begin{cor}
   \label{cor:MoreiraReisFormula}
   For $p=2$ we obtain the formula of Moreira and Reis~{\upshape\cite{mor}}
   \begin{align*}
      r(2,n)
      &=(2^{2n-1}+2^{n+1}-2^{n-1}+1)/3
      \\
      &=(2^{2n-1}+2^{n}+2^{n-1}+1)/3
      \\
      &=(2^n+1)(2^{n-1}+1)/3.
   \end{align*}
\end{cor}


\section{More occurrences of the sequence $(2,5,15,\dots)$}\label{sec3}
\subsection{Dual polar space}
Let $V$ be a vector space of dimension $2n$ over the field $\mathbb F_2$ with the standard symplectic form.
Then every maximal totally isotropic subspace of $V$ has dimension $n$.
It can be shown that every totally isotropic subspace of $V$ of dimension $n-1$ is contained in exactly three maximal totally isotropic subspaces.

Let us consider the configuration $(X,L)$ 
where the set of \emph{points} $X$ consists of the maximal totally isotropic subspaces of $V$, the set of \emph{lines} $L$ consists of the totally isotropic spaces of dimension $n-1$,
and a point $U$ lies on the line $W$ if and only if $W\subseteq U$ as vector spaces.
Thus every line contains exactly three points.
This point-line geometry is called \emph{dual polar space} $\mathrm{Sp}_{2n}(2)$.

If $n=1$, then
$X=\{
\langle\left(
\begin{smallmatrix}
   1 \\ 0
\end{smallmatrix}
\right)\rangle,
\langle\left(
\begin{smallmatrix}
   0 \\ 1
\end{smallmatrix}
\right)\rangle,
\langle\left(
\begin{smallmatrix}
   1 \\ 1
\end{smallmatrix}
\right)\rangle
\}$
and $L = \{0\}$.
If $n=2$, then $X$ and $L$ have 15 elements each and they form the so-called \emph{Cremona-Richmond} configuration, see Figure~\ref{fig:CR}.
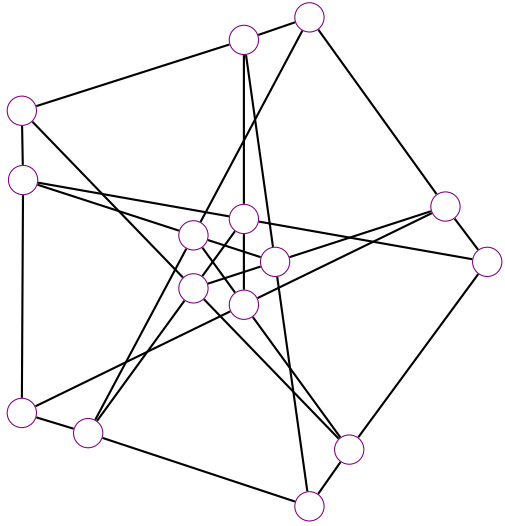
\begin{figure}
   \begin{center}
   \begin{tikzpicture}[scale=.6] 
      \node[draw,circle, violet, name=a] at (0*72:1) {};  
      \node[draw,circle, violet, name=b] at (1*72:1) {};  
      \node[draw,circle, violet, name=c] at (2*72:1) {};  
      \node[draw,circle, violet, name=d] at (3*72:1) {};  
      \node[draw,circle, violet, name=e] at (4*72:1) {};  

      \node[draw,circle, violet, name=aaa] at (0*72:5.7) {};   
      \node[draw,circle, violet, name=bbb] at (1*72:5.7) {};   
      \node[draw,circle, violet, name=ccc] at (2*72:5.7) {};   
      \node[draw,circle, violet, name=ddd] at (3*72:5.7) {};   
      \node[draw,circle, violet, name=eee] at (4*72:5.7) {};   


      \path[green, name path=aaab] (aaa)--($(b)!-8cm!(aaa)$);
      \path[red,   name path=ac] (a)--($(c)!-6cm!(a)$);
      \path [name intersections={of={aaab and ac}}];
      \coordinate (rrr) at (intersection-1);
      \node[draw, circle, violet, name=R] at (rrr) {};

      \path[green, name path=bbbc] (bbb)--($(c)!-8cm!(bbb)$);
      \path[red,   name path=bd] (b)--($(d)!-6cm!(b)$);
      \path [name intersections={of={bbbc and bd}}];
      \coordinate (sss) at (intersection-1);
      \node[draw, circle, violet, name=S] at (sss) {};

      \path[green, name path=cccd] (ccc)--($(d)!-8cm!(ccc)$);
      \path[red,   name path=ce] (c)--($(e)!-6cm!(c)$);
      \path [name intersections={of={cccd and ce}}];
      \coordinate (ttt) at (intersection-1);
      \node[draw, circle, violet, name=T] at (ttt) {};

      \path[green, name path=ddde] (ddd)--($(e)!-8cm!(ddd)$);
      \path[red,   name path=da] (d)--($(a)!-6cm!(d)$);
      \path [name intersections={of={ddde and da}}];
      \coordinate (uuu) at (intersection-1);
      \node[draw, circle, violet, name=U] at (uuu) {};

      \path[green, name path=eeea] (eee)--($(a)!-8cm!(eee)$);
      \path[red,   name path=eb] (e)--($(b)!-6cm!(e)$);
      \path [name intersections={of={eeea and eb}}];
      \coordinate (vvv) at (intersection-1);
      \node[draw, circle, violet, name=V] at (vvv) {};

      \draw[thick] (a)--(c); \draw[thick](c)--(R);
      \draw[thick] (b)--(d); \draw[thick](d)--(S);
      \draw[thick] (c)--(e); \draw[thick](e)--(T);
      \draw[thick] (d)--(a); \draw[thick](a)--(U);
      \draw[thick] (e)--(b); \draw[thick](b)--(V);

      \draw[thick] (aaa)--(T); \draw[thick](T)--(eee);
      \draw[thick] (bbb)--(U); \draw[thick](U)--(aaa);
      \draw[thick] (ccc)--(V); \draw[thick](V)--(bbb);
      \draw[thick] (ddd)--(R); \draw[thick](R)--(ccc);
      \draw[thick] (eee)--(S); \draw[thick](S)--(ddd);

      \draw[thick] (aaa)--(b); \draw[thick](b)--(R);
      \draw[thick] (bbb)--(c); \draw[thick](c)--(S);
      \draw[thick] (ccc)--(d); \draw[thick](d)--(T);
      \draw[thick] (ddd)--(e); \draw[thick](e)--(U);
      \draw[thick] (eee)--(a); \draw[thick](a)--(V);

   \end{tikzpicture}
   \end{center}
\caption{\label{fig:CR}Cremona-Richmond configuration.}
\end{figure}
\smallskip

Let $\mathbb F_2 L$ and $\mathbb F_2 X$ be the $\mathbb F_2$-vector spaces freely generated by $L$ and $X$, respectively.
Consider the map $\sigma:\mathbb F_2 L\to \mathbb F_2 X$ which sends every line of the graph to the sum of its three elements.
The dimension of the universal embedding is given by the $\mathbb F_2$-dimension of the quotient $\op{udim}:= \dim (\mathbb F_2 X/\sigma\mathbb F_2 L)$.
For $n=1$, $\op{udim}=2$; for $n=2$, $\op{udim}=5$.
This can be easily verified in the configuration of Figure~\ref{fig:CR} by the following procedure:
It is possible to mark $5$ vertices such that, whenever there is a line with already two vertices marked, the missing one is marked too, in the end, all vertices will be marked.

\noindent
For $n=3$ the configuration is unknown but it is known that it has $135$ vertices and $80$ lines.
We conjecture that \eqref{ca3}, \eqref{ca2} and \eqref{ca4} can be used to find a shortcut for the construction of these configurations.
More details can be found in \cite{carlos2}.

\subsection{Cobordism categories}
For an abelian finite group $G$ of order $n<\infty$ we consider the \emph{$G$-cobordism category} in dimension $1+1$.
Its objects are finite sequences $(g_1,\cdots, g_m)$ of elements in $G$.
Each $g\in G$ defines an $n$-fold covering of the unit circle by taking the product $G\times [0,1]$ with the identification $(h,0)\sim(h+g,1)$ for every $h\in G$.
So geometrically, the objects of the category can be represented as disjoint unions of circles.
The morphisms of the category are the gluing together of elementary cobordisms with labels as in Figure \ref{fig:elementary} up to certain identifications.
We do not yet have a short combinatorial description of all morphisms; this matter will be discussed in a forthcoming work.

\begin{figure} 
   \hspace*{\fill}
   \parbox[c]{.3\textwidth}{
   \includegraphics{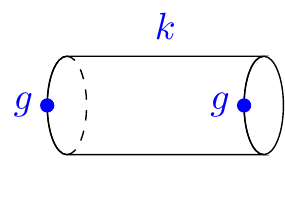}
   }
   \hspace{\fill}
   \parbox[c]{.3\textwidth}{
   \includegraphics{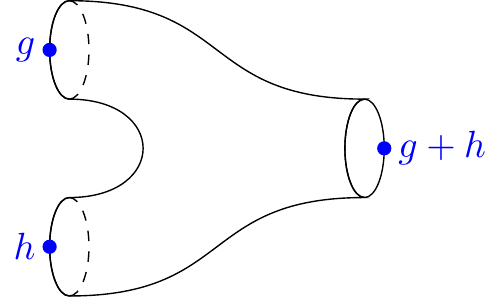}
   }
   \hspace{\fill}
   \parbox[c]{.1\textwidth}{
   \includegraphics{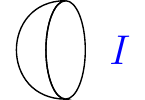}
   }
   \hspace*{\fill}
   \caption{\label{fig:elementary}Elementary components of cobordisms: cylinder, pair of pants and disc.}
\end{figure}

The importance of the $G$-cobordism category in dimension $1+1$ is its algebraic counterpart given by the definition of a $G$-Frobenius algebra, see \cite{a10,kaufmann}.
The relation between a geometric and an algebraic object is interpreted in functorial terms by what is called a $G$-equivariant topological field theory, see \cite{a10}. 
Recently, the study of the fundamental group associated to the classifying space of this category gave an interesting invariant for the group $G$.
Indeed, for a finite abelian group $G$ the first author \cite{carlos1} proves that this group is isomorphic to $\ZZ^{r(G)}:=\bigoplus_{i=1}^{r(G)}\ZZ$, 
where $r(G)$ is the cardinality of the quotient $(G\times G)/\SL(2,\ZZ)$, see \cite{carlos1}.
Therefore, for the special case $G=\ZZ_p^n$, we have $r(\ZZ_p^n)=r(p,n)$ and
Theorem~\ref{thm:formula} leads to the following theorem.
\begin{theorem}
   \label{thm:fundgroup}
   Let $G=\ZZ_p^n$ for some prime number $p$ and $n\in\NN$.
   Then the fundamental group of the classifying space of the $G$-cobordism category is isomorphic to $\ZZ^{r(G)}$, where
   \begin{align*}
      r(G)
      = r(\ZZ_p^n)
      =\frac{p^{2n-1}+p^{n+1}-p^{n-1}+p^2-p-1}{p^2-1}\,.
   \end{align*}
\end{theorem}

\subsection{Topological field theory}
The sequence \eqref{eq:sequence} appears in topological field theory (TFT) in connection with the classification of the invertible TFTs (see \cite{til3}) and more recently for the unoriented case (see \cite{loca}).
We denote by $\cob^{\ZZ_2^n}$ the $\ZZ_2^n$-cobordism category in dimension $1+1$ (see \cite{carlos1}).
A \emph{topological field theory} is a symmetric monoidal functor from $\cob^{\ZZ_2^n}$ to the category of vector spaces. We denote the category of TFTs by

\begin{equation*}
\cob^{\ZZ_2^n}-\op{SymmMon}[\cob^{\ZZ_2^n},\op{Vect}_\CC]_\ast\,.
\end{equation*}
When we take the invertible ones, we have to pass to the Picard category of the vector spaces which is identified with
$\CC^\times:=\CC\setminus\{0\}$.
Thus the category of invertible TFTs is

\begin{equation*}
   \cob^{\ZZ_2^n}-\op{SymmMon}[\cob^{\ZZ_2^n},Pic(\op{Vect}_\CC)]_\ast\,.
\end{equation*}
The category of fractions of $\cob^{\ZZ_2^n}$ is equivalent
to the fundamental group of its classifying space which has the form $\ZZ^{r(\ZZ_2^n)}$.
In \cite{carlos1}, it was proved that
$r(\ZZ_2^n)=(2^n+1)(2^{n-1}+1)/3$, that is, the sequence \eqref{eq:sequence}. Thus we have that
\begin{equation*}
\cob^{\ZZ_2^n}-\op{SymmMon}[\cob^{\ZZ_2^n},Pic(\op{Vect}_\CC)]_\ast\simeq [\ZZ^{r(\ZZ_2^n)},\CC^\times]={(\CC^\times)}^{r(\ZZ_2^n)}
\end{equation*}
where $\ast$ means based monoidal functors.


\bibliographystyle{alpha}
\bibliography{lit-SeWi-2014}

\end{document}